\documentclass[12pt]{article}
\newcommand{\spec}{\operatorname{Spec}}
\usepackage{amsmath,amssymb,amscd}
\usepackage{amsthm}

\newtheorem{thm}{\indent \sc Theorem}[section]

\newtheorem{cor}[thm]{\indent \sc Corollary}
\newtheorem{lem}[thm]{\indent \sc Lemma}

\newtheorem{rem}[thm]{\indent \sc Remark}

\usepackage[all]{xy}

\newcommand{\cF}{{\mathcal F}}

\begin{document}
\title{On the Hasse Principle for the Brauer group 
of a purely transcendental extension field in
one variable over an arbitrary field}
\author{Makoto Sakagaito}
\date{}
\renewcommand{\thefootnote}{}
\maketitle
\begin{center}
\textit{Tohoku University}
\end{center}\begin{abstract}
In this paper we show the Hasse principle for the Brauer group of a purely transcendental extension field in
one variable over an arbitrary field.
\end{abstract}

\section{Introduction}

\footnote{\textit{Key words and phrases}: Brauer group, \'{E}tale Cohomology, 
Hasse Principle}

For a field $k$, let $k_{s}$ be the separable closure of $k$ 
and $\bar{k}$ the algebraic closure of $k$.
Let $K$  be a global field (i.e., an algebraic number field or an
algebraic function field of transcendental degree one over a finite field), 
$S$  the set of all primes of $K$ and
$\widehat{K}_{\mathfrak p}$  
the completion of $K$ at $\mathfrak p\in S$. 
For a ring $A$, let  $\operatorname{Br}(A)$ be the Brauer group of $A$
(see \cite[p.141, IV, \S 2]{Me}).  
Then, the local-global map
\begin{equation*}
\operatorname{Br}(K)\to\displaystyle\prod_{\mathfrak p\in
 S}\operatorname{Br}(\widehat{K}_{\mathfrak p}) 
\end{equation*}
is injective (see \cite[Theorem 8.42 (2)]{K-K-S}).  
We call a statement of this form the 
Hasse principle. 
It is also known that the Hasse principle holds if $K$ is a purely transcendental
extension field in one variable over a perfect field $k$ (see \cite{YA}). We show that
it also holds without any assumption on  $k$. 
The following is our main theorem.

\vskip 5pt

Theorem \ref{Bha}. 
Let $k$ be an arbitrary field, $k(t)$ the purely transcendental
extension field in one variable $t$ over $k$ and $\widehat{k(t)}_{\mathfrak
 p}$ the quotient field of the completion of 
\begin{math}
\mathcal{O}_{\mathbb{P}_{k}^{1}, \mathfrak p}.
\end{math}
Then, the local-global map
\begin{equation*}
\operatorname{Br}(k(t))\to\displaystyle\prod_{\substack{\mathfrak
 p\in\mathbb{P}_{k}^{1 }\\ \operatorname{ht}(\mathfrak p)=1}}\operatorname{Br}(\widehat{k(t)}_{\mathfrak
 p})
\end{equation*}
is injective. 

Moreover, if $k$ is a separably closed field, the Hasse principle for
the Brauer group of any algebraic function fields in one variable over
$k$ is shown by using \cite[Corollaire (5.8)]{G} as in the case of
Theorem \ref{Bha}.

For the defference between the case of perfect fields and Theorem 
\ref{Bha}, see Remark \ref{remark:difference}.

\vskip 5pt

\section{Notations}
For a field $k$ and a Galois extension field $k^{\prime}$ of $k$,
$G(k^{\prime}/k)$ denotes the Galois group of $k^{\prime}/k$ and $k_{s}$
denotes the separable closure of $k$. 
We denote $G(k_{s}/k)$ by $G_{k}$ and the category of (discrete)
$G_{k}$-modules (cf, \cite[p.10, I]{S1}) by $G_{k}$-mod. 
For a discrete $G(k^{\prime}/k)$-module $A$ (but the action is continuous) and a positive integer $q$, 
$\operatorname{H}^{q}(k^{\prime}/k, A)$ denotes 
the $q$-th cohomology group of $G(k^{\prime}/k)$ with coefficients in $A$
(see \cite[p.10, I, \S 2]{S1}).  We put  $\operatorname{H}^{q}(k,
A)=\operatorname{H}^{q}(k_{s}/k, A)$. 
\begin{math}
\operatorname{Res}: \operatorname{H}^{p}(k, A)\to
\operatorname{H}^{p}(k^{\prime}, A)
\end{math}
denotes the restriction homomorphism.
For a group $G$, we put
$G_{q}=\{ g\in G~ |~ g^{q}=1\}$ and $X(G)$ the group of characters of $G$.

For a scheme $X$, 
$X^{(i)}$ is the set of points of codimension $i$ and
$X_{(i)}$ is the set of points of dimension $i$. 
We denote the \'{e}tale site (resp. finite \'{e}tale site) on $X$ by
$X_{et}$ (resp. $X_{fet}$) and
the category of sheaves over $X_{et}$ (resp. $X_{fet}$) by 
$\mathbb{S}_{X_{et}}$ (resp. $\mathbb{S}_{X_{fet}}$). For
$\cF\in\mathbb{S}_{X_{et}}$ (resp. $\mathbb{S}_{X_{fet}}$), 
we denote the $q$-th cohomology group of
$X_{et}$ ($X_{fet}$) with values in $\cF$ by
$\operatorname{H}^{q}_{et}(X, \cF)$ or even simply
$\operatorname{H}^{q}(X, \cF)$
(resp. $\operatorname{H}^{q}_{fet}(X, \cF)$). 
If $Y\subset X$ is a closed subscheme, we denote 
the $q$-th local (\'{e}tale) cohomology with support in $Y$ by 
$\operatorname{H}^{q}_Y(X, \cF)$.   
For an integral scheme $X$ and
$\mathfrak p\in X^{(1)}$, let $R(X)$ be the function field of $X$, 
$\mathcal{O}_{X, \mathfrak p}$ the local ring at 
$\mathfrak p$  of $X$, $\widehat{\mathcal{O}}_{X, \mathfrak p}$
the completion of $\mathcal{O}_{X, \mathfrak p}$, 
$\widehat{R(X)}_{\mathfrak p}$ its quotient field , 
$\widetilde{\mathcal{O}}_{X, \mathfrak p}$
the Henselization of $\mathcal{O}_{X, \mathfrak p}$ ,
$\widetilde{R(X)}_{\mathfrak p}$ its quotient field , $\mathcal{O}_{X,
\bar{\mathfrak p}}$ the
strictly Henselization of 
$\mathcal{O}_{X, \mathfrak p}$ and 
$R(X)_{\bar{\mathfrak p}}$ its
quotient field.

\section{Main theorem}

\begin{thm}\upshape\label{H-B}
Let $X$ be a 1-dimensional connected regular scheme, $K$ its quotient field. Then
\begin{equation}
\xymatrix{
0\ar[r]
&\operatorname{Br}(X)\ar[r]
&\operatorname{Br}(K)\ar[r]
&\displaystyle
\prod_{\mathfrak p\in X^{(1)}}
\operatorname{Br}(\widetilde{R(X)}_{\mathfrak p})/
\operatorname{Br}(\widetilde{\mathcal{O}}_{X, \mathfrak p})
}
\end{equation}
is exact.
\end{thm}

\begin{proof}
Suppose that $B$ is a discrete valuation ring, $L$ is its quotient
 field, $Y=\spec{B}$ and $Z=Y\setminus\spec{L}=\{\mathfrak p\}$. Then we have the exact sequence
\begin{equation}
\operatorname{H}^{p}(Y, \mathbb{G}_{m})\to 
\operatorname{H}^{p}(\spec{L}, \mathbb{G}_{m})\to
\operatorname{H}^{p+1}_{Z}(Y, \mathbb{G}_{m})
\end{equation}
by \cite[p.92, III, Proposition 1.25]{Me} and 
\begin{math}
\operatorname{H}^{2}(Y, \mathbb{G}_{m})\to \operatorname{H}^{2}(\spec L, \mathbb{G}_{m})
\end{math}
is injective by \cite[p.145, IV, \S2]{Me}. Moreover we have
\begin{equation}\label{sec}
\operatorname{H}^{p}_{Z}(Y,
 \mathbb{G}_{m})\simeq\operatorname{H}^{p}_{\{\mathfrak p\}}
(\spec(\widetilde{\mathcal{O}}_{Y,
 \mathfrak p}), \mathbb{G}_{m})
\end{equation}
by \cite[p.93, III, Corollary 1.28]{Me}. Moreover, the diagram
\begin{equation*}
\xymatrix{
\operatorname{Br}(K)/
\operatorname{Br}(\mathcal{O}_{X, \mathfrak p})\ar[r]\ar[d]
&\operatorname{Br}(\widetilde{R(X)}_{\mathfrak p})/
\operatorname{Br}(\widetilde{\mathcal{O}}_{\mathfrak p})\ar[d] \\
\operatorname{H}^{3}_{\{\mathfrak p\}}
(\spec(\mathcal{O}_{X, \mathfrak p}), \mathbb{G}_{m})
\ar[r]_{cf, (\ref{sec})}^{\simeq}
&\operatorname{H}^{3}_{\{\mathfrak p\}}
(\spec(\widetilde{\mathcal{O}}_{X, \mathfrak p}), \mathbb{G}_{m})
}
\end{equation*}
is commutative. Therefore 
\begin{equation*}
\operatorname{Br}(K)/\operatorname{Br}(\mathcal{O}_{X, \mathfrak p})\to
\operatorname{Br}(\widetilde{R(X)}_{\mathfrak p})/\operatorname{Br}
(\widetilde{\mathcal{O}}_{X, \mathfrak p})
\end{equation*}
is injective. So the statement follows from \cite[p.77, II, Proposition 2.3]{G}.
\end{proof}
\begin{lem}\upshape\label{com}
Let $A$ be a Henselian discrete valuation ring, $K$ its quotient field ,
 $k$ its residue field and $K_{nr}$ its maximal unramified extension. Then
\begin{equation*}
\operatorname{H}^p(\spec(A),\mathit{g}_{*}(\mathbb{G}_{m}))
=\operatorname{H}^p(K_{nr}/K, (K_{nr})^{*})
\end{equation*}
for any $p>0$ and the sequence 
\begin{equation}\label{hen}
0\to \operatorname{H}^{p}(\spec(A), \mathbb{G}_{m})\to 
\operatorname{H}^p(K_{nr}/K, (K_{nr})^{*})\to
\operatorname{H}^{p}(k, \mathbb{Z})\to 0
\end{equation}
is exact.
\end{lem}
\begin{proof}
Let 
%$k$ be the residue field of $A$ and 
$\mathit{i}$: $\spec(k)\to \spec(A)$ be the
natural map. Then, $\mathit{i}_{*}$ is 
exact. Let $(set)$ be the class of all separated etale morphisms and $f$:
$X_{et}\to X_{set}$ the continuous morphism which is induced by identity map on $X$. 
Then $f_{*}$ is exact by \cite[p.112, (b) of Examples 3.4]{Me}. Let
 $(fet)$ be the class of all finite
 etale morphisms and $f^{\prime}$: $X_{set}\to X_{fet}$  the
 continuous morphism
 which is induced by identity map on $X$. 

Let $Y\to X$ be a separated
 etale morphism with $Y$ connected, $R(Y)$ the ring of rational functions of
 $Y$, $A \to B$ the normalization of $A$ in $R(Y)$ and $X^{\prime}=\spec(B)$. Then
 $R(Y)/K$ is a finite separable extension and $Y$ is an open subscheme of
 $X^{\prime}$ by \cite[p.29, I, Theorem 3.20]{Me}. Moreover
 $X^{\prime}\to X$ is finite by \cite[p.4, I, Proposition
 1.1]{Me}. Then, since $A$ is a Henselian discrete valuation ring,
 $B$ is a Henselian discrete valuation ring by \cite[p.33, I, (b) of Theorem 4.2]{Me} and
 \cite[p.34, I, Corollary 4.3]{Me}. Also $R(X^{\prime})/R(X)$ is
 an unramfied extension. Therefore $f^{\prime}_{*}$ is exact by \cite[p.111,
 III, Proposition 3.3]{Me}. So $f^{\prime}_{*}\circ f_{*}$ is exact and
\begin{equation*}
\operatorname{H}^{p}_{fet}(X, (f^{\prime}\circ f)_{*}(\cF))
\simeq
\operatorname{H}^{p}_{et}(X, \cF)
\end{equation*}
for any $\cF\in \mathbb{S}_{X_{et}}$. 

We have the isomorphism 
\begin{math}
G_{K}\text{-}\mathbf{mod}\simeq \mathbb{S}_{\spec(K)_{et}}
\end{math}
by \cite[p.53, II.\S1,Theorem1.9]{Me}. Let the functor $N$ be defined as
\begin{equation*}
(G_{K}\text{-}\mathbf{mod})\ni M\longmapsto
 M^{\operatorname{Gal}(K_{s}/K_{nr})}\in (G_{k}\text{-}\mathbf{mod})
\end{equation*}
and 
\begin{math}
N^{\prime}: \mathbb{S}_{\spec(K)_{et}}\to \mathbb{S}_{\spec(k)_{et}}
\end{math}
the functor which corresponds to $N$. Let 
\begin{math}
Y^{\prime\prime}\in X_{fet}
\end{math}
be connected. 
Moreover, let $K^{\prime\prime}=R(Y^{\prime\prime})$ and $k^{\prime\prime}$ the finite
 extension field of $k$ which corresponds to the closed point of
 $Y^{\prime\prime}$. Then 
\begin{equation*}
N^{\prime}(F)(\spec(k^{\prime\prime}))=F(\spec(K^{\prime\prime}))
\end{equation*}
for $F\in \mathbb{S}_{\spec(K)_{et}}$ because
\begin{equation*}
G(K_{nr}/K^{\prime\prime})\simeq G_{k^{\prime\prime}},
~~G(K_{nr}/K^{\prime\prime})\simeq 
G_{K^{\prime\prime}}/G_{K_{nr}}.
\end{equation*}
Therefore the diagram
 %by Lemma \ref{gror} $(2)$ and the fact that
 %
\begin{equation*}
\label{g-l}
\xymatrix{G_{K}\text{-}\mathbf{mod}\ar[d]_{N}\ar@{}[r]|{\simeq}
& \mathbb{S}_{\spec(K)_{et}}\ar[r]^{f^{\prime}_{*}\circ 
f_{*}\circ\mathit{g}_{*}}\ar[d]^{N^{\prime}}
& \mathbb{S}_{X_{fet}}\\
G_{k}\text{-}\mathbf{mod}\ar@{}[r]|{\simeq}
& \mathbb{S}_{\spec(k)_{et}}\ar[ur]_{f^{\prime}_{*}\circ
f_{*}\circ\mathit{i}_{*}}
&.
}
\end{equation*}
is commutative.
So 
\begin{align*}
\operatorname{H}^p_{et}(X,\mathit{g}_{*}(\mathbb{G}_{m}))
&=\operatorname{H}^p_{fet}(X,f^{'}\circ f\circ\mathit{g}_{*}(\mathbb{G}_{m}))\\
&=\operatorname{H}^p_{fet}(X,f^{'}\circ f\circ
\mathit{i}_{*}(N^{\prime}(\mathbb{G}_{m}))) \\
&=\operatorname{H}^p_{et}(X,\mathit{i}_{*}(N^{\prime}(\mathbb{G}_{m}))) \\
&=\operatorname{H}^p_{et}(\spec(k), N^{\prime}(\mathbb{G}_{m}))\\
&=\operatorname{H}^p( k, (K_{nr})^{*})
=\operatorname{H}^p( K_{nr}/K, (K_{nr})^{*}).
\end{align*}
If we want to show where we consider the sheaf $\mathbb{G}_{m}$, we use
the notation such as $\mathbb{G}_{m, A}$. 
Then the exact sequence (\ref{hen}) follows from the exact sequence of sheaves
\begin{equation*}
0\to \mathbb{G}_{m, A}\to 
g_{*}(\mathbb{G}_{m, K})\to
i_{*}(\mathbb{Z})\to
0
\end{equation*}
(cf, \cite[p.106, III, Example 2.22]{Me}). So the proof is complete.
%see Lemma \ref{com}.
\end{proof}

\begin{cor}\upshape\label{GB}
Consider the situation of Theorem \ref{H-B} and
\begin{equation*}
\operatorname{Br}_{un}(X)
=\operatorname{Ker}\left(\operatorname{Br}(K)\stackrel{\operatorname{Res}}{\to}
\displaystyle\prod_{\mathfrak p\in X_{(0)}}
\operatorname{Br}(\widetilde{R(X)}_{\bar{\mathfrak p}})\right).
\end{equation*}
Then the sequence
\begin{equation}\label{Fad2}
0\to\operatorname{Br}(X)\to
\operatorname{Br}_{un}(X)\to
\displaystyle\prod_{\substack{\mathfrak p\in X^{(1)}}}
\operatorname{X}(G_{\kappa(\mathfrak p)})
\end{equation}
is exact.
\end{cor}
\begin{proof}
%Suppose that $X=\spec A$ where a discrete valuation ring $A$. Then 
%
It follow from \cite[p.76, II, Corollaire 2.2]{G} and \cite[p.147, IV,
 Proposition 2.11 (b)]{Me} that
\begin{math}
\operatorname{Br}(\mathcal{O}_{X, \mathfrak p})\subset
\operatorname{Br}_{un}(\spec(\mathcal{O}_{X, \mathfrak p})).
\end{math}
So the sequence
\begin{equation*}
0\to \operatorname{Br}(\mathcal{O}_{X, \mathfrak p})
\to \operatorname{Br}_{un}(\spec(\mathcal{O}_{X, \mathfrak p}))
\to \operatorname{Br}(\widetilde{R(X)}_{\mathfrak p})/
\operatorname{Br}(\widetilde{\mathcal{O}}_{X, \mathfrak p})
\end{equation*}
is exact by Theorem \ref{H-B}. Moreover, 
\begin{math}
 \operatorname{Br}(\widetilde{R(X)}_{\mathfrak p})/
\operatorname{Br}(\widetilde{\mathcal{O}}_{X, \mathfrak p})
\simeq X(G_{\kappa(\mathfrak p)})
\end{math}
by Lemma \ref{com}. Therefore the sequence 
\begin{equation}\label{A}
0\to \operatorname{Br}(\mathcal{O}_{X, \mathfrak p})
\to \operatorname{Br}_{un}(\spec(\mathcal{O}_{X, \mathfrak p}))
\to X(G_{\kappa(\mathfrak p)})
\end{equation}
is exact. So the
 statement follows from (\ref{A}) and \cite[p.77, II, Proposition 2.3]{G}. 
\end{proof}
\begin{rem}\upshape
\begin{enumerate}
\item Suppose that X is a regular algebraic curve over a field $k$. If $k$ is
 perfect,
\begin{math}
\operatorname{Br}_{un}(X)=\operatorname{Br}(K)
\end{math}
by \cite[p.80, II, 3.3]{S1}. If $(n, \operatorname{ch}(k))=1$,
\begin{math}
\operatorname{Br}_{un}(X)_{n}=\operatorname{Br}(K)_{n}
\end{math}
by \cite[p.111, Appendix, \S2, (2.2)]{S1}.
\item Corollary \ref{GB} is true even if $\operatorname{dim}X\neq 1$ 
because 
\begin{equation*}
\operatorname{H}^2(X, \mathit{g}_*(\mathbb{G}_{m,
 K}))=\operatorname{Ker}\left(\operatorname{Br}(K)\stackrel{\operatorname{Res}}{\to}\displaystyle\prod_{x\in X_{(0)}}\operatorname{Br}(K_{\bar{x}})\right)
\end{equation*}
where $\mathit{g}$ :
\begin{math}
\spec{K}\to X
\end{math}
is the generic point of $X$.
\end{enumerate}
\end{rem}

\begin{thm}\label{Bha}\upshape
Let $k$ be an arbitrary field $k$ and $k(x)$ the purely transcendental
extension field in one variable $x$ over $k$. Then, the local-global map
\begin{equation*}
\operatorname{Br}(k(x))\to
\displaystyle
\prod_{\mathfrak
 p\in\mathbb{P}_{k}^{1 (1)}}\operatorname{Br}(\widehat{k(x)}_{\mathfrak
 p})
\end{equation*}
is injective. 
\end{thm}
\begin{proof}
By using the facts \cite[proof of Theorem 1]{K-K} and \cite[p.674, \S3.4, Lemma
 16]{K}, we see that 
\begin{math}
\operatorname{Br}(\widetilde{k(x)}_{\mathfrak
 p})\simeq
\operatorname{Br}(\widehat{k(x)}_{\mathfrak
 p}).
\end{math}
So it is sufficient for the proof of the statement to prove that 
\begin{equation*}
\operatorname{Br}(k(x))\to\displaystyle
\prod_{\mathfrak
 p\in\mathbb{P}_{k}^{1 (1)}}\operatorname{Br}(\widetilde{k(x)}_{\mathfrak
 p})
\end{equation*}
is injective. We denote the point which corresponds to 
\begin{math}
(\frac{1}{x}) \in \spec(k[\frac{1}{x}])\subset \mathbb{P}_{k}^{1}
\end{math}
by $\infty$. Then, by Theorem \ref{H-B}, 
\begin{align*}
&\operatorname{Ker}\left(
\operatorname{Br}(k(x))\to\displaystyle
\prod_{\mathfrak p\in\mathbb{P}_{k}^{1 (1)}}
\operatorname{Br}(\widetilde{k(x)}_{\mathfrak p})
\right) \\
\subset
&\operatorname{Ker}\left(
\operatorname{Br}(k(x))\to
\displaystyle
\prod_{\mathfrak p\in \left((\mathbb{P}_{k}^{1})^{(1)}\setminus \infty\right)}
\operatorname{Br}(\widetilde{R(\mathbb{P}_{k}^{1})}_{\mathfrak p})/
\operatorname{Br}(\widetilde{\mathcal{O}}_{\mathbb{P}_{k}^{1}, \mathfrak p})
\right)\\
=&\operatorname{Br}(k[x]).
\end{align*}
Moreover
\begin{align*}
\operatorname{Ker}\left(
\operatorname{Br}(k(x))\to\displaystyle
\prod_{\mathfrak p\in\mathbb{P}_{k}^{1 (1)}}
\operatorname{Br}(\widetilde{k(x)}_{\mathfrak p})
\right)\subset\operatorname{Ker} 
\left(
\operatorname{Br}(k[x])\to\operatorname{Br}(k(x))\to
\operatorname{Br}(\widetilde{R(\mathbb{P}_{k}^{1})}_{\infty})
\right)
\\
\end{align*}
and
\begin{math}
\operatorname{Ker} 
\left(
\operatorname{Br}(k[x])\to\operatorname{Br}(k(x))\to
\operatorname{Br}(\widetilde{R(\mathbb{P}_{k}^{1})}_{\infty})
\right)=0
\end{math}
by \cite[p.153, IV, Exercise 2.20 (d)]{Me} or \cite{YS}. Therefore 
\begin{equation*}
\operatorname{Ker}\left(
\operatorname{Br}(k(x))\to\displaystyle
\prod_{\mathfrak p\in\mathbb{P}_{k}^{1 (1)}}
\operatorname{Br}(\widetilde{k(x)}_{\mathfrak p})
\right)=0.
\end{equation*}
So the statement follows.
\end{proof}

\begin{cor}\upshape
Let $X$ be an algebraic curve over a seperably closed field such that
 regular and proper. Then, the
 local-global map
\begin{equation*}
\operatorname{Br}(R(X))\to\displaystyle\prod_{\mathfrak
 p\in X^{(1)}}\operatorname{Br}(\widehat{R(X)}_{\mathfrak
 p})
\end{equation*}
is injective.
\end{cor}
\begin{proof}
The statement follows from Theorem \ref{H-B} and \cite[III, Corollary 5.8]{G}.
\end{proof}

\begin{rem}
\label{remark:difference}
\upshape
If $k$ is perfect, Theorem \ref{Bha} is proved by using the exact
 sequence
\begin{equation}\label{Fad}
0\to \operatorname{Br}(\mathbb{P}_{k}^{1})\to 
\operatorname{Br}(k(x))\to
\bigoplus_{\mathfrak p\in \mathbb{P}_{k}^{1 (1)}}X(G_{\kappa(\mathfrak p)})
\end{equation}
in \cite{YA}. But it is unknown fact whether (\ref{Fad}) is exact or not in
 the case where $k$ is not perfect and Theorem \ref{Bha} has not been proved.
The sequence (\ref{Fad2}) is exact in Corollary \ref{GB},
 but the sequence (\ref{Fad}) is not exact in the case where $k$ is not
 perfect as follows.

It is known that $k$ is perfect if and only if
\begin{math}
\operatorname{Br}(k)=\operatorname{Br}(k[x]) 
\end{math}
(cf, \cite[p.389, Theorem 7.5]{A-G}). So $\operatorname{Br}(k[x])\neq 0$ in
 the case where $k$ is the separable closure of an imperfect field and 
$\operatorname{Br}(k(x))\neq 0$ because 
\begin{math}
\operatorname{Br}(k[x])\subset \operatorname{Br}(k(x)).
\end{math}
On the other hand,
\begin{math}
X(G_{\kappa(\mathfrak p)})=\{1\}
\end{math}
and
\begin{math}
\operatorname{Br}(\mathbb{P}_{k}^{1})=\operatorname{Br}(k)=\{0\}.
\end{math}
So the sequence (\ref{Fad}) is not exact.
\end{rem}

\textsc{Acknowledgment}. The author would like to thank Professors Nobuo Tsuzuki,
Takao Yamazaki and Akihiko Yukie for valuable comments.

\begin{quote}
\textit{Present Address}: \\
\textsc{Makoto Sakagaito\\
Mathematical Institute\\
Tohoku University\\
Sendai 980-8578\\
Japan
} \\
\textit{e-mail}: sa4m12@math.tohoku.ac.jp
\end{quote}
\end{document}